\documentclass{amsart}

\usepackage{amsmath,amsthm, amscd, amssymb, amsfonts}
\usepackage{epsfig,pb-diagram}
\input xy
\xyoption{all}
\usepackage[usenames]{color}

\newcommand{\mapsfrom}{\mathrel{\mbox{$\leftarrow\joinrel\mapstochar$\,}}}

\newcommand{\Mo}{{\mathcal M}}

\newcommand{\ca}{{\mathcal C}}

\newcommand{\Bc}{{\mathcal B}}

\newcommand{\cM}{{\mathcal M}}
\newcommand{\cO}{{\mathcal O}}

\newcommand{\Z}{{\mathbb Z}}
\newcommand{\N}{{\mathbb N}}

\newcommand{\id}{\mbox{\rm id\,}}

\newcommand{\ku}{\Bbbk}

\newcommand{\Ind}{\mbox{\rm Ind\,}}

\newcommand{\Vect}{\mbox{\rm Vec}}

\newcommand\Rep{\operatorname{Rep}}

\newcommand{\ord}{\mathop{\rm ord}}

\newcommand{\ad}{\mbox{\rm ad\,}}

\theoremstyle{plain}
\numberwithin{equation}{section}

\newtheorem{teo}{Theorem}[section]
\newtheorem{lema}[teo]{Lemma}
\newtheorem{cor}[teo]{Corollary}
\newtheorem{prop}[teo]{Proposition}

\theoremstyle{definition}
\newtheorem{defi}[teo]{Definition}
\newtheorem{exa}[teo]{Example}

\theoremstyle{remark}

\newtheorem{rmk}[teo]{Remark}

\def\pf{\begin{proof}}
\def\epf{\end{proof}}

\theoremstyle{remark}

\begin{document}

\title[De-equivariantization of Hopf algebras]
{De-equivariantization of Hopf algebras}
\author[Angiono, Galindo and Pereira]{Iv\'an Angiono, C\'esar Galindo and Mariana Pereira}

\address{I. A.: FaMAF-CIEM (CONICET), Universidad Nacional de C\'ordoba,
Medina A\-llen\-de s/n, Ciudad Universitaria, 5000 C\' ordoba, Rep\'
ublica Argentina.}

\email{angiono@famaf.unc.edu.ar}

\address{C. G.: Departamento de matem\'aticas, Universidad de los Andes, Carrera 1 N. 18A - ´
10, Bogot\'a, Colombia.}

\email{cn.galindo1116@uniandes.edu.co}

\address{M. P.: Instituto de Matem\'atica y Estad\'istica Rafael Laguardia. Facultad  de Ingenier\'ia. Universidad de la Rep\'ublica. J.H.Reissig 565, CP 11.300 , Montevideo, Uruguay.}

\email{maripere@fing.edu.uy}

\thanks{\noindent 2010 \emph{Mathematics Subject Classification.}
16W30, 18D10, 19D23. \newline The work of I. A. was partially
supported by CONICET, FONCyT-ANPCyT and Secyt (UNC).
M.P. is grateful for the support from the grant ANII FCE 2007-059}

\begin{abstract}
We study the de-equivariantization of a Hopf algebra by an affine
group scheme and we apply Tannakian techniques in order to realize
it as the tensor category of comodules over a coquasi-bialgebra. As
an application we construct a family of coquasi-Hopf algebras
$A(H,G,\Phi)$ attached to a coradically-graded pointed Hopf algebra
$H$ and some extra data.
\end{abstract}

\maketitle

\section*{Introduction}

Actions of groups over abelian categories have been studied in recent years
with the purpose of constructing, describing and studying
categories with symmetries. For example, Gaitsgory \cite{G}
introduced the notion of the action of an affine group scheme $G$
over a $\mathbb C$-linear abelian category $\ca$ and the category of
$G$-equivariant objects $\ca^G$, called the equivariantization of
$\ca$ by $G$. The category $\ca^G$ has an action of $\Rep(G)$ and
the category of Hecke eigen-objects in $\ca^G$ is again $\ca$. In
general, if $\Rep(G)$ acts on an abelian category $\mathcal{C}$,
then the category of Hecke eigen-objects in $\mathcal{C}$ is called
the de-equivariantization of $\mathcal{C}$ by $G$.

Equivariantization and de-equivariantization are standard techniques
in theory of fusion categories \cite{DGNO} and have been applied in
geometric Langlands program \cite{FG} and quantum groups \cite{ArG}.

Now, if $\ca$ is a tensor category and the action of $\Rep(G)$ over
$\ca$ is tensorial, then the de-equivariantization has a natural
tensor structure. A special but very important type of tensor
categories are those equivalent to the category of corepresentations
of a Hopf algebra, which include representations of algebraic
groups, quantum groups, compact groups, etc. If $\ca$ is the
category of comodules (or finite dimensional modules) over a Hopf
algebra, then  $\ca^G\to \ca\to \Vect$ is a fiber functor on $\ca^G$
(where $\ca^G\to \ca$ is the forgetful functor) and by Tannakian
duality $\ca^G$ is the category of comodules over Hopf algebra.
Thus, the family of Hopf algebras is closed under equivariantization, in the sense that we obtain new categories which are equivalent to categories of corepresentations of Hopf algebras.
This is not the case for the de-equivariantization process since the de-equivariantization of comodules
over a Hopf algebra is not always equivalent to the category of
corepresentations over a Hopf algebra (see Subsection
\ref{subseccion ejemplo}, for concrete examples). However, under
some mild conditions, it is always the category of
corepresentations over a coquasi-bialgebra. As a consequence there
exist coquasi-Hopf algebras not twist equivalent to Hopf algebras,
which admit an equivariantization equivalent to a Hopf algebra. This
phenomenon was used  in \cite{EG2} to relate the Drinfeld doubles of
some quasi-Hopf algebras with small quantum groups, and in \cite{A}
in order to classify the family of basic quasi-Hopf algebras with
cyclic group of one-dimensional representations, under some mild
conditions.

In this paper we study the de-equivariantization of the category of
comodules over a Hopf algebra by an affine group scheme and apply
Tannakian techniques to realize the de-equivariantization as the
tensor category of comodules over a coquasi-bialgebra.

We apply the construction to interpret the central
extensions of Hopf algebras as a particular example, and an additional
application to the context of pointed finite tensor categories,
extending the family of examples obtained in \cite{EG2}, \cite{Ge}, \cite{A}.

The organization of the paper is the following. In Section
\ref{section:preliminaries} we recall the definitions related with
the main construction of this paper. First, the relation between
affine group schemes and commutative algebras, then co-quasi
bialgebras, and finally the center of a tensor category. In Section
\ref{section:De-equivariantization} we build a co-quasi Hopf algebra
which represents the tensor category obtained as the
de-equivariantization of the category of co-representations of a
Hopf algebra. To do this, we consider central braided Hopf
bialgebras, which are in correspondence with inclusions of tensor
categories of comodules over Hopf algebras with certain
factorization through the center, making emphasis on the case of
algebras of functions over an affine group (in particular, over
finite groups). We then obtain the corresponding coquasi-Hopf
algebra representing a de-equivariantization over the comodules of a
Hopf algebra by a Tannakian reconstruction. Finally, Section
\ref{section:applications} contains some applications of the
previous results. The main one is the case of finite-dimensional
pointed Hopf algebras, which gives place to a general construction of
pointed coquasi-Hopf algebras, and consequently finite pointed
tensor categories.

\section{Preliminaries}\label{section:preliminaries}

In this section we recall some definitions and results on Hopf
algebras, affine group schemes and coquasi-Hopf algebras. For further reading on these topics we the reader to \cite{Mont}, \cite{Wat} and \cite{Sch2} respectively. Throughout
the paper we work over an arbitrary field $\ku$. Algebras and
coalgebras are always defined over $\ku$. For a coalgebra
$(C,\Delta,\varepsilon)$ we shall use Sweedler's notation
omitting the sum symbol, that is $\Delta(c)= c_1\otimes c_2$ for all
$c\in C$. Similarly if $(M,\lambda)$ is a left $C$-comodule, then
$\lambda(m)= m_{-1}\otimes m_0\in C\otimes M$ for all $m\in M$. The
category of left $C$-comodules shall be denoted by $^C\Mo$.

\subsection{Affine group scheme and commutative Hopf algebras}
Let $\ku$-$\mathcal{A}lg$ denote the category of commutative
$\ku$-algebras and $\mathcal{G}rp$ the category of groups. An affine
group scheme over $\ku$ is a representable functor
$G:\ku$-$\mathcal{A}lg \to \mathcal{G}rp$. By Yoneda's lemma the
commutative algebra that represents $G$ is unique up to
isomorphisms, and we shall denote it by $\mathcal{O}(G)$. The group
structures on $G(A)$, $A\in \ku$-$\mathcal{A}lg$, determine natural
transformations
\begin{align*}
    m:& \; G\times G\to G,\\
1:& \; Sp(\ku)\to G,\\
i:& \; G\to G,
\end{align*}and they define algebra maps
\begin{align*}
    \Delta:& \; \mathcal{O}(G) \to \mathcal{O}(G)\otimes \mathcal{O}(G),\\
\varepsilon:& \; \mathcal{O}(G)\to \ku,\\
\mathcal{S}:& \; \mathcal{O}(G)\to \mathcal{O}(G),
\end{align*}
that give a Hopf algebra structure on $\mathcal{O}(G)$. Conversely,
if $K$ is a commutative Hopf algebra, then
$Spec(K):\ku\text{-}\mathcal{A}lg\to \mathcal{S}et, A\mapsto
Alg(K,A)$ is an affine group scheme with group structure given by
the convolution product and this defines an anti-equivalence of
categories between affine groups schemes over $\ku$ and commutative
Hopf algebras over $\ku$.

Under this equivalence the category of representations of $G$ is
equivalent  to the category of  $\mathcal{O}(G)$-comodules, and
quasi-coherent sheaves on $G$ are $\mathcal{O}(G)$-modules.

\subsection{Coquasi-bialgebras}

A \emph{coquasi-bialgebra} $(H,m,u,\omega ,\Delta ,\varepsilon )$ is
a coalgebra $(H,\Delta ,\varepsilon )$ together with
coalgebra morphisms:
\begin{itemize}
 \item the multiplication $m:H\otimes H\longrightarrow
H$ (denoted $ m(h\otimes g)=hg$),
 \item the unit $u:\ku \longrightarrow H$ (where we call
$ u(1)=1_{H} $),
\end{itemize}
and a convolution invertible element $\omega \in (H\otimes H\otimes
H)^{\ast }$ such that for all $h,g,k,l\in H$:
\begin{eqnarray}
h_{1}(g_{1}k_{1})\omega (h_{2},g_{2},k_{2}) &=&\omega
(h_{1},g_{1},k_{1})(h_{2}g_{2})k_{2}  \label{asociat multipl}
\\ 1_{H}h &=&h1_{H}=h
\\ \omega(h_{1}g_{1},k_{1},l_{1})\omega (h_{2},g_{2},k_{2}l_{2}) &=& \omega
(h_{1},g_{1},k_{1})\label{cocycle omega}
\\ && \qquad \omega(h_{2},g_{2}k_{2},l_{1}) \omega (g_{3},k_{3},l_{2})\nonumber
\\ \omega (h,1_{H},g) &=&\varepsilon (h)\varepsilon (g). \end{eqnarray}

Note that $\omega (1_{H},h,g)=\omega (h,g,1_{H})=\varepsilon
(h)\varepsilon (g)$ for each $g,h\in H$.
\medskip

A \emph{coquasi-Hopf algebra} is a coquasi-bialgebra $H$ endowed
with a coalgebra anti-homomorphism $\mathcal{S}:H\longrightarrow H$
(the \emph{antipode}) and elements $\alpha $, $\beta \in H^{\ast }$
satisfying, for all $h\in H$:
\begin{eqnarray} \mathcal{S}(h_{1})\alpha (h_{2})h_{3} &=&\alpha (h)1_{H}  \label{SalfaId}
\\ h_{1}\beta (h_{2})\mathcal{S}(h_{3}) &=&\beta (h)1_{H}  \label{IdbetaS}
\\  \label{omega anihileaza S} \varepsilon (h)&=&\omega (h_{1}\beta (h_{2}),\mathcal{S}(h_{3}),\alpha (h_{4})h_{5})
\\ &=&\omega ^{-1}(\mathcal{S}(h_{1}),\alpha (h_{2})h_{3}\beta
(h_{4}),\mathcal{S}(h_{5})).\nonumber
\end{eqnarray}
\smallskip

The  category of left $H$-comodules $\ {}^{H} \mathcal{M}$ is rigid
and monoidal, where the tensor product is over the base field and
the comodule structure of the tensor product is the codiagonal one.
The associator is given by
\begin{eqnarray*} \phi _{U,V,W} &:&(U\otimes V)\otimes
W\longrightarrow U\otimes (V\otimes W) \\ \phi _{U,V,W}((u\otimes
v)\otimes w) &=&\omega (u_{-1},v_{-1},w_{-1})u_{0}\otimes
(v_{0}\otimes w_{0}) \end{eqnarray*} for $u\in U$, $v\in V$, $w\in
W$ and $U,V,W\in {}^{H}\mathcal{M}$. The dual coactions are given by
$\mathcal{S}$ and $\mathcal{S}^{-1}$, as in the case of Hopf
algebras.

\subsection{The center construction and the category of Yetter-Drinfeld modules}
The center construction produces a braided monoidal category
$\mathcal{Z}(\ca)$ from any monoidal category $\ca$, see \cite{ka}.
The objects of $\mathcal{Z}(\ca)$ are pairs $(Y, c_{-,Y})$, where $Y
\in \ca$ and $c_{X,Y} : X \otimes Y \to Y \otimes X$ are
isomorphisms natural in $X$ satisfying $c_{X\otimes Y,Z} = (c_{XZ}
\otimes \id_Y )(\id_X\otimes c_{Y, Z})$ and $c_{I,Y} = \id_Y$, for
all $X,Y,Z \in \ca$. The braided monoidal structure is given in the
following way:
\begin{itemize}
  \item the tensor product is $ (Y, c_{-,Y}) \otimes (Z, c_{-,Z}) =
(Y\otimes Z, c_{-,Y\otimes Z})$, where
$$c_{X,Y \otimes
Z} = (\id_Y \otimes c_{X,Z})(c_{X,Y} \otimes \id_Z) : X \otimes Y
\otimes Z \to Y \otimes Z \otimes X,$$ for all $X \in \ca$,
  \item the identity element is $(I, c_{-,I})$, $c_{Z,I} = \id_Z$
  \item the braiding is the morphism $c_{X,Y}$.
\end{itemize}

Let $H$ be a Hopf algebra with bijective antipode. We shall denote
by $^H\Mo$ the tensor category of left $H$-comodules. The category
$\mathcal{Z}(^H\Mo)$ is braided equivalent to the category
$^H_H\mathcal{YD}$ of left-left Yetter-Drinfeld modules, whose
objects are left $H$-comodules and left $H$-modules $M$ satisfying
the condition
\begin{equation}\label{Yetter-Drinfeld condition}
(h_1\rightharpoonup m)_{-1}h_2\otimes (h_1 \rightharpoonup m)_0 =
h_1m_{-1}\otimes h_2\rightharpoonup m_0
\end{equation}for all $m\in M$, $h\in H$.  A Yetter-Drinfeld module $N$ becomes
an object in $\mathcal{Z}(^H\Mo)$ by \[ c_{M,N}(m\otimes n)=
m_{-1}\rightharpoonup n\otimes m_0,\]and inverse
$c_{M,N}^{-1}(n\otimes m)= m_0\otimes
\mathcal{S}^{-1}(m_1)\rightharpoonup n$.

\section{De-equivariantization of Hopf algebras}\label{section:De-equivariantization}

\subsection{Central inclusion and braided central Hopf subalgebras}

Let $H$ be a Hopf algebra with bijective antipode. Let $G$ be an
affine group scheme over $\ku$ and $\mathcal{O}(G)$ the Hopf algebra
of regular functions over $G$.

A \textbf{central inclusion} of $G$ in $H$ is a braided monoidal inclusion
$\iota: \Rep(G) \hookrightarrow \mathcal{Z}(^H\Mo)\cong \
^H_H\mathcal{YD}$, such that the braiding of $^H_H\mathcal{YD}$
restricts to the usual symmetric braiding of $\Rep(G)$, and the
composition $\Rep(G) \hookrightarrow\ ^H_H\mathcal{YD} \to\ ^H\Mo$
gives an inclusion.

In order to describe in Hopf-theoretical terms the central
inclusions, we need the following concept.


\begin{defi}
Let $H$ be a Hopf algebra. A \emph{braided central Hopf subalgebra}
of $H$ is a pair $(K,r)$, where $K\subset H$ is a Hopf subalgebra,
and $r:H\otimes K\to \ku$ is a bilinear form such that:
\begin{align}
r(hh',k)&=r(h',k_1)r(h,k_2), \label{cuasi-trenza 1}
\\ r(h,kk')&=r(h_1,k)r(h_2,k'), \label{cuasi-trenza 2}
\\ r(h,1)&=\varepsilon(h), \quad  r(1,k)=\varepsilon(k),\label{cuasi-trenza 3}\\
r(h_1,k_1)k_2h_2&=h_1k_1r(h_2,k_2), \label{casi commutativa} \\
r(k,k')&=\varepsilon(kk'), \label{conmutatividad}
\end{align}
for all $k,k'\in K$, $h,h'\in H$.
\end{defi}

\begin{rmk}\label{rmk definicion braided central}

\begin{enumerate}
\item The conditions \eqref{cuasi-trenza 1}, \eqref{cuasi-trenza 2}, \eqref{cuasi-trenza 3},
say that $r:H\otimes K\to \ku$ is a Hopf skew pairing, so in
particular $r$ has a convolution-inverse
$$r^{-1}(h,k)=r(h,S(k)), \  \   \  ( h\in H, k\in K).$$

  \item The algebra $K$ is commutative by \eqref{casi commutativa} and \eqref{conmutatividad}.
  \item For all $V\in\- ^H\Mo, W\in\- ^K\Mo$, the map $r$ defines a
  natural isomorphism $c_{V,W}:V\otimes W \to W\otimes V, v\otimes w\mapsto r(v_{-1},w_{-1})w_0\otimes v_0$ in $^H\Mo$, and  these isomorphisms define a braided inclusion $^K\Mo\to \mathcal{Z}(^H\Mo)=\- ^H_H\mathcal{YD}$.
  \item The condition \eqref{conmutatividad} implies that $K$ is a
  commutative algebra in $\mathcal{Z}(^H\Mo)$.
\end{enumerate}
\end{rmk}

For example, any central Hopf subalgebra $K\subset H$ is braided
central with $r=\varepsilon_H\otimes\varepsilon_K$. Conversely, if
$K\subset H$ is a braided central Hopf subalgebra with
$r=\varepsilon_H\otimes\varepsilon_K$ then $K$ is a central Hopf
subalgebra.
\smallskip

\begin{lema}\label{r factoriza}
Let $K\subset H$ be a braided central Hopf subalgebra. Then
\begin{equation}\label{r pasa al cociente}
r(xh,k)=r(hx,k)=\varepsilon(x)r(h,k),
\end{equation}
for all $x, k\in K,  h\in H$,
\end{lema}
\begin{proof}
It follows from conditions \eqref{cuasi-trenza 1} and
\eqref{conmutatividad}.
\end{proof}

The following result exhibits the relevance of braided central Hopf
subalgebras.

\begin{teo}\label{calculo r matriz}
Let $H$ be a Hopf algebra and $K\subset H$ a commutative Hopf
subalgebra. Then the following set of data are equivalent:

\begin{enumerate}
  \item A map $r:H\otimes K\to \ku$ such that $(K,r)$ is a braided central Hopf subalgebra of $H$.
  \item A braided monoidal functor $F:\ ^K\Mo \to \mathcal{Z}(^H\Mo)=\ ^H_H\mathcal{YD}$ such that
  the composition with the forgetful functor $\mathcal{Z}(^H\Mo)=\ ^H_H\mathcal{YD} \to\ ^H\Mo$ is an inclusion.
  \item A Hopf algebra map $\gamma: K\to (H^\circ)^{cop}$ with
  $\gamma(k)|_{K}=\varepsilon$  and
  $$\langle \gamma( k_1),h_1\rangle k_2h_2=h_1k_1\langle \gamma(k_2),h_2\rangle$$ for all $h \in H, k\in K$
  ( $H^\circ$ denotes the finite dual Hopf algebra).
\end{enumerate}
\end{teo}
\begin{proof}
(1) $\Rightarrow$ (2) Let $M\in\ ^K\Mo$, then the map
$\rightharpoonup : H\otimes M\to M, h\otimes m\mapsto
r(h,m_{-1})m_0$, defines a structure of $H$-module, that satisfies
the Yetter-Drinfeld compatibility by \eqref{casi commutativa}.

\noindent (2) $\Rightarrow$ (3) Since every comodule is a colimit of
finite dimensional comodules, the image of the monoidal functor
$^K\Mo \to\ ^H_H\mathcal{YD}\to\ _H\Mo$ lives in the tensor subcategory of  $\underline{ _H\Mo}$ of $H$-modules that are colimits of finite dimensional $H$-modules, then the monoidal functor
$^K\Mo \to\ ^H_H\mathcal{YD}\to\ \underline{ _H\Mo}\cong\
^{(H^\circ)^{cop}}\Mo$ induces a unique Hopf algebra map $\gamma:K
\to (H^\circ)^{cop}$ given by
$$h\rightharpoonup m = \langle \gamma(m_{-1}), h\rangle m_0,$$ for
all $h\in H$, $m\in M$ and $M\in \ ^K\Mo$.

It is enough to prove that
$$h_1m_{-2}\langle\gamma(m_{-1}),h_2\rangle\otimes m_0 = m_{-1}h_2 \langle
\gamma(m_{-2}),h_1\rangle\otimes m_0,$$ for all $m\in M$, $M\in \
^K\Mo$. Indeed, \eqref{Yetter-Drinfeld condition} implies that
\begin{align*}
    h_1m_{-2}\langle \gamma(m_{-1},h_2\rangle\otimes m_0 &= h_1m_{-2}\otimes \langle \gamma(m_{-1},h_2\rangle m_0\\
    &=  h_1m_{-2}\otimes h_2\rightharpoonup m_0\\
    &= (h_1\rightharpoonup m)_{-1}h_2\otimes (h_1 \rightharpoonup m)_0\\
    &= m_{-1}h_2\otimes \langle \gamma(m_{-2},h_1\rangle m_0\\
     &= m_{-1}h_2 \langle \gamma(m_{-2}),h_1\rangle\otimes m_0.
\end{align*}

\noindent (3) $\Rightarrow$ (1) The map $r(h,k)= \langle
\gamma(k),h\rangle$ defines a braided central structure over $K$.
\end{proof}
The following result in an immediate consequence of Theorem
\ref{calculo r matriz}.
\begin{cor}
Let $H$ be a Hopf algebra. There exist a bijective correspondence
between central inclusions of $G$ in $H$ and braided central Hopf
subalgebras $K$ of $H$ such that $K\cong\mathcal{O}(G)$ as Hopf
algebras.\qed
\end{cor}

\subsection{De-equivariantization of a Hopf algebra by an affine group scheme}

Let $H$ be a Hopf algebra and $G$ be an affine group scheme. Let
$K\subset H$ a braided central Hopf subalgebra with
$K=\mathcal{O}(G)$.

The algebra $\mathcal{O}(G)$ is a commutative algebra in the
symmetric category $\Rep(G)$, and thus a commutative algebra in the
braided tensor category $^H_H\mathcal{YD}$ (see  Remark \ref{rmk
definicion braided central} item (4)). Therefore, the algebra
$\mathcal{O}(G)$ is braided commutative.

We define the de-equivariantization $\ ^H\Mo(G)$ of $\ ^H\Mo$ by
$G$, as the category of $H$-equivariant sheaves on $G$, that is the
category of left $\mathcal{O}(G)$-modules in $^H\Mo$.

Now, the category  $\ ^H_{G}\Mo_{G}$ of $\mathcal{O}(G)$-bimodules
in $^H\Mo$ is a tensor category with the tensor product
$M\otimes_{\mathcal{O}(G)} N$. We shall see in the next proposition
that this tensor product induces a monoidal structure on $\
^H\Mo(G)$.

\begin{prop}
Let $V\in\ ^H\Mo(G)$ with left $\mathcal{O}(G)$-module structure
$\rightharpoonup : \mathcal{O}(G)\otimes V \to V$ and left
$H$-comodule structure $\lambda: V\to \mathcal{O}(G)\otimes V,
v\mapsto v_{-1}\otimes v_0$. The map $\leftharpoonup :V\otimes
\mathcal{O}(G) \to V, v\leftharpoonup x= r(v_{-1},x_1)x_2
\rightharpoonup v_0$, makes $V$ an object  in $\ ^H_{G}\Mo_{G}$.
This rule defines a fully faithful strict monoidal functor from
$^H\Mo(G)$ to $\ ^H_{G}\Mo_{G}$.
\end{prop}
\begin{proof}Let $V\in\ ^H\Mo(G)$ with left $\mathcal{O}(G)$-module
structure $\rightharpoonup : \mathcal{O}(G)\otimes V \to V$ and left
$H$-comodule structure $\lambda: V\to \mathcal{O}(G)\otimes V,
v\mapsto v_{-1}\otimes v_0$.

 (1) The map $\leftharpoonup :V\otimes \mathcal{O}(G) \to
V$ defines a right $\mathcal{O}(G)$-module structure: for any $v\in
V$ and $x,y\in \mathcal{O}(G)$,
\begin{align*}
(v\leftharpoonup x)\leftharpoonup y &= (r(v_{-1},x_1)x_2 \rightharpoonup v_0)\leftharpoonup y \\
 &= r(v_{-1},x_1)r((x_2 \rightharpoonup v_0)_{-1},y_1)y_2\rightharpoonup  (x_2 \rightharpoonup v_0)_0\\
 &= r(v_{-2},x_1)r(x_2 v_{-1},y_1)y_2\rightharpoonup  (x_3 \rightharpoonup v_0)\\
 &= r(v_{-2},x_1)r( v_{-1},y_1)y_2x_2 \rightharpoonup v_0\\
 &= r(v_{-1},y_1x_1)y_2x_2 \rightharpoonup v_0\\
 &= v \leftharpoonup (xy),
\\    v\leftharpoonup 1 &= r(v_{-1},1)v_0= \varepsilon(v_{-1})v_0= v.
\end{align*}

(2) The map $\leftharpoonup :V\otimes \mathcal{O}(G) \to V$ is a
morphism in $^H\Mo$:
\begin{align*}
    (v\leftharpoonup x)_{-1}\otimes (v\leftharpoonup x)_0 &= r(v_{-2},x)v_{-1}\otimes
    v_0,\\
    v_{-1}x_{1}\otimes v_0\leftharpoonup x_2 &= v_{-2}x_{1}\otimes r(v_{-1},x_{2})v_0\\
     &= r(v_{-1},x_{2}) v_{-2}x_{1}\otimes v_0\\
     &= r(v_{-2},x_{1}) x_{2}v_{-1}\otimes v_0\\
     &= r(v_{-2},x_{1}) \varepsilon(x_{2})v_{-1}\otimes v_0\\
     &= r(v_{-2},x) v_{-1}\otimes v_0.
\end{align*}

(3) The maps $\leftharpoonup$ and $\rightharpoonup$ commute:
\begin{align*}
    (x\rightharpoonup v)\leftharpoonup y &= r((x\rightharpoonup
    v)_{-1},y_1)y_2 \rightharpoonup (x\rightharpoonup v)_0\\
    &= r(x_1v_{-1},y_1)y_2 \rightharpoonup x_2\rightharpoonup v_0\\
    &= \varepsilon(x_1) r(v_{-1},y_1)y_2 \rightharpoonup x_2\rightharpoonup v_0\\
    &=  r(v_{-1},y_1)y_2 \rightharpoonup x\rightharpoonup v_0\\
    &= x\rightharpoonup (v\leftharpoonup y).
\end{align*}

(4) Let $f: V\to W$ a morphism in $\ ^H\Mo(G)$,  to see that $f:
V\to W$ is a morphism in $\ ^H_{G}\Mo_{G}$ is enough to prove that
$f$ is a right $\mathcal{O}(G)$-module morphism,
\begin{align*}
f(v \leftharpoonup x)= f(r(v_1,x_1)x_2 \rightharpoonup v_0) &=
r(v_1,x_1)x_2 \rightharpoonup f(v_0)\\ &= r(f(v)_{-1},x_1)x_2
\rightharpoonup f(v)_0\\
&= f(v)\leftharpoonup x.
\end{align*}for all $x\in \mathcal{O}(G), v\in V$.

Therefore we have a well-defined fully faithful functor from $\
^H\Mo(G)$  to $\ ^H_{G}\Mo_{G}$.

Let $V, W \in \ ^H\Mo(G)$, $v\in V, w\in W, x\in \mathcal{O}(G)$,
the calculation
\begin{align*}
r((v\otimes_{\mathcal{O}(G)} w)_{-1},x_1)x_2 \rightharpoonup
(v\otimes_{\mathcal{O}(G)} w)_{0} &=r(v_{-1}w_{-1},x_1) (x_2
\rightharpoonup
v_{0})\otimes_{\mathcal{O}(G)} w_{0}\\
&=r(w_{-1},x_1)r(v_{-1},x_2) (x_3 \rightharpoonup
v_{0})\otimes_{\mathcal{O}(G)} w_{0}\\
&=r(w_{-1},x_1) (v_0 \leftharpoonup x_{2})\otimes_{\mathcal{O}(G)}
w_{0}\\
&=r(w_{-1},x_1) v_0 \otimes_{\mathcal{O}(G)}x_2 \rightharpoonup
w_{0}\\
&=v\otimes_{\mathcal{O}(G)} (w\leftharpoonup x).
\end{align*}
proves that the right $\mathcal{O}(G)$-action of
$V\otimes_{\mathcal{O}(G)} W$ is  induced by the left
$\mathcal{O}(G)$-action; in other words, $\otimes_{\mathcal{O}(G)}$
defines a monoidal structure on $\ ^H\Mo(G)$ such that $\ ^H\Mo(G)$
is a tensor subcategory of $\ ^H_{G}\Mo_{G}$.
\end{proof}

\begin{defi}
Let $H$ be a Hopf algebra and $G$ an affine group scheme, with a
central inclusion of $G$ in $H$. The category $\ ^H\Mo(G)$ with the
monoidal structure $\otimes_{\mathcal{O}(G)}$ is called the
\emph{de-equivariantization} of $H$ by $G$.
\end{defi}

\subsection{Tannakian reconstruction of $\ ^H\Mo(G)$}

Let $H$ be a Hopf algebra and $\mathcal{O}(G)\subset H$ be a braided
central inclusion of $G$ in $H$. We shall say that the central
inclusion of $G$ in $H$ is \emph{cleft} if there exists a
convolution invertible $\mathcal{O}(G)$-linear map $\pi:H\to
\mathcal{O}(G)$; such a map is called a \emph{cointegral}.

Lemma \ref{r factoriza} implies that  $r : H
\otimes \mathcal{O}(G) \to \ku$ induces a well-defined map
$$ H/\mathcal{O}(G)^+H\otimes \mathcal{O}(G) \to \ku,\qquad \overline{h}\otimes k \mapsto r(h, k),$$
which we will denote again by $r$ (here,
$\mathcal{O}(G)^+=\ker(\varepsilon)$ is the augmentation ideal). The
goal of this section is to prove the following result.

\begin{teo}\label{main result}
Let $H$ be a Hopf algebra and $(\mathcal{O}(G),r)$ a cleft braided
central Hopf subalgebra with cointegral $\pi$ such that
$\varepsilon\pi=\varepsilon$ and $\pi(1)=1$. Then the quotient
coalgebra $Q:=H/\mathcal{O}(G)^+H$ is a coquasi-bialgebra with
multiplication and associator given by:
\begin{align}\label{multiplicacion coquasi}
m(a\otimes b) &=  \overline{j(a_{1})
    j(b)_{1}} r(a_{2},
    \pi(j(b)_{2}))
\\\label{associator coquasi}
\omega(a\otimes b\otimes c) &= r(a,\pi(j(b)_{1}j(c)_{1}))
  r(j(b)_2,\pi(j(c)_2)),
\end{align}
where $a,b,c\in Q$ and $j:Q\to H$, $q\mapsto \pi^{-1}(q_1)q_2$.
There is a monoidal equivalence between $\ ^H\Mo(G)$ and $^Q\Mo$.
\end{teo}

Before to give the proof of the Theorem, we want to explain briefly
the Tannakian reconstruction principle that we shall use. Let
$C$ be a coalgebra and $^C\Mo$ be the category of left
$C$-comodules. Assume that $^C\Mo$ has a monoidal structure
$\overline{\otimes}:\ ^C\Mo\times\ ^C\Mo\to\ ^C\Mo$,
$\alpha_{V,W,Z}: (V\overline{\otimes} W)\overline{\otimes} Z \to
V\overline{\otimes} (W\overline{\otimes} Z)$, $1\overline{\otimes}
V= V\overline{\otimes} 1$ such that the underlying functor $^C\Mo\to
\text{Vect}_{\ku}$ is a strict quasi-monoidal functor, \emph{i.e.},
$V\overline{\otimes} W= V\otimes_k W$ and $1=k$ as vector spaces,
then $C$ has a coquasi-bialgebra structure $(m, \omega)$ given by
\begin{align}\label{reconstruccion multiplicacion}
    m(a,b)= (a\otimes b)_{-1}\varepsilon((a\otimes b)_0),\\
    \omega(a,b,c)=\varepsilon(\alpha(a\otimes b\otimes c)) \label{reconstruccion asociador}
\end{align}
and the monoidal structure on $^C\Mo$ defined by the
coquasi-bialgebra structure coincides with the monoidal structure
$(\overline{\otimes}, \alpha, 1)$.

\begin{proof}

From now on we fix a cointegral $\pi$ such that
$\varepsilon\pi=\varepsilon$ and $\pi(1)=1$. Let $Q=
H/\mathcal{O}(G)^+H$ be the quotient coalgebra of $H$, then by
\cite[Theorem 2.4]{DMR} and \cite[Theorem II]{Schneider} the
functors
\begin{align}\label{equivalencia coalgebra de-equi}
    \widehat{\mathcal{V}}:\ ^H\Mo(G) &\to\- ^Q\Mo\\
                              M &\mapsto \overline{M}= M/K^+M\\
                              H\square_Q V &\mapsfrom V,
\end{align}
define a category equivalence, where $\overline{M}=
M/\mathcal{O}(G)^+M$ is a left $Q$-comodule with
$\overline{m}_{-1}\otimes \overline{m}_0= \overline{m_{-1}}\otimes
\overline{m_0}$, and $H\square_Q V=\{\sum h\otimes v| \sum
h_1\otimes h_2\otimes v = \sum h\otimes v_{-1}\otimes v_0\}\in \
^H\Mo(G)$ has as left $H$-comodule and a left
$\mathcal{O}(G)$-module structures the ones induced by the left
tensor factor.

By \cite[Lemma 3.3.5]{Sch2}, for all $M, N\in\ ^H\Mo(G)$ we have a
linear isomorphism
\begin{align*}
    \xi_{M,N}^\pi:\overline{M}\otimes \overline{N} &\to \overline{M\otimes_{\mathcal{O}(G)} N}\\
                  \overline{m}\otimes \overline{n}&\mapsto
                  \overline{m\otimes_{\mathcal{O}(G)} \pi^{-1}(n_{-1})n_0}\\
                  r(m_{-1},\pi(n_{-1})_{1})\overline{\pi(n_{-1})_2m_0}\otimes \overline{n_0}=\overline{m\pi(n_{-1})}\otimes \overline{n_{0}} &\mapsfrom \overline{m\otimes_{\mathcal{O}(G)} n}
\end{align*}
such that the functor $(\mathcal{V},\xi^\pi):\ ^H\Mo(G)\to
\Vect_\ku$, $M\mapsto \overline{M}:= M/K^+M$ is  quasi-tensor. Then
using the equivalence $\eqref{equivalencia coalgebra de-equi}$, the
category $^Q\Mo$ has a (unique) monoidal structure such the
$\widehat{\mathcal{V}}$ is a monoidal equivalence and the following
diagram of functors commutes
$$
\begin{diagram}
\node{^H\Mo(G)} \arrow{se,r}{(\mathcal{V},\xi)}
\arrow[2]{e,l}{\widehat{\mathcal{V}}} \node{} \node{^Q\Mo} \arrow{sw,r}{\mathcal{U}}\\
\node{}\node{\text{\Vect}_\ku}\node{}
\end{diagram}
$$
Consequently the underlying functor $\mathcal{U}$ becomes an strict
quasi-monoidal functor and we can apply Tannakian
reconstruction.

A natural section $j:Q\to H$ for the canonical projection $\nu_H:
H\to Q$ is given by
\begin{equation}\label{eq:section}
j(\overline{h})=\pi^{-1}(h_{-1})h_{0}.
\end{equation}

Fix a cointegral $\pi: H\to \mathcal{O}(G)$ such that $\pi(1)=1$,
and define $j$ as in \eqref{eq:section}.

The $Q$-comodule structure on $Q\otimes Q$ is:
\begin{align*}
    (\overline{m}\otimes \overline{n})_{-1}\otimes (\overline{m}\otimes
    \overline{n})_0 &= \xi(\overline{m}\otimes
    \overline{n})_{-1}\otimes \xi^{-1}(\xi(\overline{m}\otimes
    \overline{n})_0)\\
     &= (\overline{m}\otimes j(\overline{n}))_{-1}\otimes \xi^{-1}(\overline{m}\otimes
     j(\overline{n}))\\
     &= \overline{m_{-1}j(\overline{n})_{-1}}\otimes \xi^{-1}(\overline{m}_0\otimes
     j(\overline{n})_0)\\
     &= \overline{m_{-1}j(\overline{n})_{-1}}\otimes r(\overline{m}_{0,-1},\pi(j(\overline{n})_{0,-1})_1)
     \\ & \qquad \quad \overline{\pi(j(\overline{n})_{0,-1})_2m}_{00} \otimes
     j(\overline{n})_{00}\\
     &= \overline{m_{-2}j(\overline{n})_{-2}}\otimes r(\overline{m}_{-1},
     \pi(j(\overline{n})_{-1})_1)
     \\ & \qquad \quad \overline{\pi(j(\overline{n})_{-1})_2m}_{0} \otimes
     j(\overline{n})_{0}\\
      &= \overline{j(\overline{m}_{-2})j(\overline{n})_{-2}}\otimes r(\overline{m}_{-1},
     \pi(j(\overline{n})_{-1})_1)
     \\ & \qquad \quad \overline{\pi(j(\overline{n})_{-1})_2m}_{0} \otimes
     j(\overline{n})_{0},
\end{align*}
for all $m,n \in H$, $\overline{m}, \overline{n}\in Q$. Now,
applying the formula \eqref{reconstruccion multiplicacion}, we have
\begin{align*}
    m(\overline{m}\otimes \overline{n}) &= \overline{j(\overline{m}_{-2})
    j(\overline{n})_{-2}} r(\overline{m}_{-1},
    \pi(j(\overline{n})_{-1})_1)\varepsilon(\overline{\pi(j(\overline{n})_{-1})_2m}_{0})
    \varepsilon( j(\overline{n})_{0})\\
    &= \overline{j(\overline{m}_{-2})
    j(\overline{n})_{-1}} r(\overline{m}_{-1},
    \pi(j(\overline{n})_{0})_1)\varepsilon(\overline{\pi(j(\overline{n})_{0})_2m}_{0})\\
    &= \overline{j(\overline{m}_{-1})
    j(\overline{n})_{-1}} r(\overline{m}_{0},
    \pi(j(\overline{n})_{0})),
\end{align*}that is $$m(a\otimes b)=  \overline{j(a_{1})
    j(b)_{1}} r(a_{2},
    \pi(j(b)_{2})).$$

The constraint of associativity of $^Q\Mo$, is defined by the
commutativity of the diagram
$$
\begin{diagram}
\node{\overline{L}\otimes \overline{M}\otimes \overline{N}}
\arrow{s,l}{\xi_{L,M}\otimes\id}
\arrow[2]{e,l}{\alpha_{\overline{L},\overline{M},\overline{N}}}
\node{}
\node{\overline{L}\otimes \overline{M}\otimes \overline{N}} \arrow{s,r}{\id\otimes \xi_{M,N}}\\
\node{\overline{L\otimes_{\mathcal{O}(G)}M}\otimes \overline{N}}
\arrow{se,r}{\xi_{L\otimes_{\mathcal{O}(G)} M,
N}}\node{}\node{\overline{L}\otimes
\overline{M\otimes_{\mathcal{O}(G)}N}}
\arrow{sw,r}{\xi_{L,M\otimes_{\mathcal{O}(G)}N}}\\
\node{}\node{\overline{L\otimes_{\mathcal{O}(G)}
M\otimes_{\mathcal{O}(G)}N}}\node{}
\end{diagram}
$$

Hence,
\begin{align*}
    \alpha(\overline{l}\otimes \overline{m}\otimes \overline{n}) &=
    \id\otimes\xi^{-1}_{M,N}\circ \xi^{-1}_{L,M\otimes_{\mathcal{O}(G)}N} \circ \xi_{L\otimes_{\mathcal{O}(G)}M,N}\circ \xi_{L,M}\otimes\id
    (\overline{l}\otimes \overline{m}\otimes \overline{n})\\
  &=  \id\otimes\xi^{-1}_{M,N}\circ
  \xi^{-1}_{L,M\otimes_{\mathcal{O}(G)}N} (\overline{l\otimes j(\overline{m})\otimes
  j(\overline{n})})\\
 &=  \id\otimes\xi^{-1}_{M,N}(r(l_{-1},\pi((j(\overline{m})\otimes
  j(\overline{n}))_{-1})_{1})
  \\ & \qquad \quad \overline{\pi((j(\overline{m})\otimes
  j(\overline{n}))_{-1})_{2}l_0}\otimes \overline{j(\overline{m})_0}\otimes
  \overline{j(\overline{n})_0})\\
  &= r(l_{-1},\pi(j(\overline{m})_{-1}j(\overline{n})_{-1})_{1}) \id\otimes
  \\ & \qquad \quad\xi^{-1}_{M,N}(  \overline{\pi(j(\overline{m})_{-1}
  j(\overline{n})_{-1})_{2}l_0}\otimes\overline{j(\overline{m})_0}\otimes
  \overline{j(\overline{n})_0})\\
  &= r(l_{-1},\pi(j(\overline{m})_{-1}j(\overline{n})_{-1})_{1})
  r(j(\overline{m})_{0,-1},\pi(j(\overline{n})_{0,-1})_1)
  \\ & \qquad \quad
  \overline{\pi(j(\overline{m})_{-1}
  j(\overline{n})_{-1})_{2}l_0}\otimes
  \overline{ \pi(j(\overline{n})_{0,-1})_2 j(\overline{m})_{00}}\otimes
  \overline{j(\overline{n})_{00}}\\
  &= r(l_{-1},\pi(j(\overline{m})_{-2}j(\overline{n})_{-2})_{1})
  r(j(\overline{m})_{-1},\pi(j(\overline{n})_{-1})_1)
  \\ & \qquad \quad
  \overline{\pi(j(\overline{m})_{-1}
  j(\overline{n})_{-1})_{2}l_0}\otimes
  \overline{ \pi(j(\overline{n})_{-1})_2 j(\overline{m})_{0}}\otimes
  \overline{j(\overline{n})_{0}}
\end{align*}
for all $l\in L,m\in M,n\in N $. Applying the formula
\eqref{reconstruccion asociador},

\begin{align*}
    \omega(\overline{l}\otimes \overline{m}\otimes \overline{n})
    &= r(l_{-1},\pi(j(\overline{m})_{-2}j(\overline{n})_{-2})_{1})
  r(j(\overline{m})_{-1},\pi(j(\overline{n})_{-1})_1)\\
   &
  \varepsilon(\overline{\pi(j(\overline{m})_{-1}
  j(\overline{n})_{-1})_{2}l_0}) \varepsilon(
  \overline{ \pi(j(\overline{n})_{-1})_2
  j(\overline{m})_{0}})\varepsilon(\overline{j(\overline{n})_{0}})\\
  &= r(l,\pi(j(\overline{m})_{-1}j(\overline{n})_{-1}))
  r(j(\overline{m})_0,\pi(j(\overline{n})_0)),
\end{align*}that is $$\omega(a\otimes b\otimes c)=
r(a,\pi(j(b)_{1}j(c)_{1}))
  r(j(b)_2,\pi(j(c)_2))$$for all $a,b,c\in Q$.

\end{proof}

\begin{rmk}
\begin{enumerate}
  \item If $\pi:H\to K$ is any integral then
$\pi'(h):=\pi(h_1)\varepsilon\pi^{-1}(h_2) $ is again an integral
such that $\varepsilon\pi'=\varepsilon$.
\item If $\pi:H\to K$ is an integral, then $\pi(1)\in
H^\times$, and $\pi'(h):= \pi(h)/\pi(1)$ is again an integral such
that $\pi'(1)=1$.
  \item If $H$ is finite dimensional Hopf algebra $H$, every  Hopf
  subalgebra $K\subset H$ admits an integral $\pi: H\to K$.
\end{enumerate}

\end{rmk}

\begin{prop}
If $H$ is finite dimensional and $G$ is a constant finite algebraic
group, then the coquasi-bialgebra $Q$ defined in Theorem \ref{main
result} admits a coquasi-Hopf algebra structure.
\end{prop}
\begin{proof}
Since $G$ is a constant finite group it follows  by \cite[Theorem
4.18]{DGNO} that the de-equivariantization is a rigid monoidal
category. Since $Q$ is a quotient of $H$, $Q$ is a  finite
dimensional and by \cite[Theorem 3.1]{Sch3}, $Q$ is a coquasi-Hopf
algebra.

\end{proof}

\section{Applications}\label{section:applications}

In the last part of this work we will apply the results of the Section
\ref{section:De-equivariantization} to some particular cases. First we consider the
category of $G$-graded vector spaces, for some group $G$. Second, we
look at quotient of Hopf algebras by central Hopf subalgebras, and
view them as a de-equivariantization. Finally we study a family of
pointed finite-dimensional coquasi-Hopf algebras, whose dual
algebras are a generalization of the quasi-Hopf algebras $A(H,s)$ in
\cite{A}.

\subsection{Baby example}\label{subsection:baby example}

Let $\Gamma$ be a discrete group, $G\subset \mathcal{Z}(\Gamma)$ a
central subgroup of $\Gamma$, and $r: \Gamma\times G\to \ku^*$ a
bicharacter such that $r|_{G\times G}=1$. Then the pair $(\ku G, r)$
is a braided central Hopf subalgebra of $\ku \Gamma$.

We shall fix a set of representatives of the right cosets of $G$ in
$\Gamma$, $Q \subset G$. Thus every element $\gamma\in \Gamma$ has a
unique factorization $\gamma= gq$, $g\in G$,  $q\in Q$. We assume
$e\in Q$. The uniqueness of the factorization $\Gamma = GQ$ implies
that there are well defined maps
$$\cdot:Q\times Q \to Q,\qquad \theta:Q\times Q\to G,$$
determined by the conditions
$$ pq=\theta(p,q)p\cdot q,\qquad p,q \in Q.$$
The map $\theta$ is a 2-cocycle $\theta\in Z^2(\Gamma/G, G)$ where
$\Gamma/G$ acts trivially over $G$, since $G$ is a central subgroup
of $\Gamma$.

We define a map $\pi:\Gamma\to G$, $\gamma\mapsto x$, where $x\in G$
is the unique element such that $\gamma=xp$ with $p\in Q$, and
$j:Q\to \Gamma$ is the inclusion. Now by Theorem \ref{main result},
the de-equivariantization is defined as follows. Let $K$ be the
quotient group $\Gamma/G$, then the group algebra $\ku K$ with the
3-cocycle
$$\omega(u,v,w)= r(u,\theta(v,w))$$ is a coquasi-Hopf algebra and $^{\ku K}\Mo$ is
tensor equivalent to $\ ^{\ku \Gamma}\Mo(\widehat{G})$, the
de-equivariantization of $\ ^{\ku \Gamma}\Mo$ by the affine group
scheme $\widehat{G}(-)=Alg(\ku G,-)$.

Now, we will explain how this construction determines the same data
of \cite[Example 2.2.6]{A}. If $\Gamma$ is abelian, the map
$r:\Gamma\times G\to \ku^*$ defines a group morphism $T:G\to
\widehat{\Gamma}$, $x\mapsto r(-,x)$ such that $\langle
T(x'),x\rangle = r(x,x')=1$, for all $x,x'\in G$, thus it defines an
inclusion of $\Vect_G$ as a Tannakian subcategory of
$\mathcal{Z}(\Vect_\Gamma)$, and the 3-cocycle over $K$ is:
$$\omega(u,v,w)= r(u,\theta(v,w))=\langle T(\theta(v,w)),u\rangle.$$

\subsection{Second example: Central extension of Hopf algebras}

Let $H$ be a Hopf algebra and $(K,r)$ a braided central Hopf
subalgebra, if $r(h,x)=\varepsilon(hx)$ for all $h\in H, x\in K$,
then $K\subset H$ is a central Hopf subalgebra and this defines a
central inclusion of the group scheme $G=Spec(K)$ in $H$. Also since
$K$ is central, $K^+H$ is a Hopf ideal and  $Q=H/K^+H$ is a quotient
Hopf algebra of $H$.

\begin{prop}\label{prop:central extensions}
Let $H$ be a Hopf algebra and $K\subset H$ a cleft central Hopf
subalgebra, then the de-equivariantization of $^H\Mo$ by $G=Spec(K)$
is tensor equivalent to the tensor category of comodules over the
Hopf algebra $Q=H/K^+H$.
\end{prop}
\begin{proof}
The central Hopf subalgebra $K$ is braided central with
$r(h,k)=\varepsilon(hk)$ for all $h\in H,k\in K$. Then the product
and coassociator in the coquasi-bialgebra defined in Theorem
\ref{main result} are
\begin{align*}
    m(\overline{a}\otimes \overline{b}) &=  \overline{j(\overline{a_{1}})
    j(\overline{b})_{1}} r(\overline{a_{2}},
    \pi(j(\overline{b})_{2}))\\
    &=  \overline{j(\overline{a}_{1})
    j(\overline{b})_{1}} \varepsilon(\overline{a}_{2})\varepsilon(
    \pi(j(\overline{b})_{2}))\\
     &= \overline{j(\overline{a})
    j(\overline{b})}=  \overline{j(\overline{a})
    j(\overline{b})}=\overline{ab},\\
\omega(\overline{a}\otimes \overline{b}\otimes \overline{c}) &=
r(\overline{a},\pi(j(\overline{b})_{1}j(\overline{c})_{1}))
  r(j(\overline{b})_2,\pi(j(\overline{c})_2))\\
  &= \varepsilon(\overline{a})\varepsilon(\pi(j(\overline{b})_{1}j(\overline{c})_{1}))
  \varepsilon(j(\overline{b})_2)\varepsilon(\pi(j(\overline{c})_2))\\
  &= \varepsilon(\overline{a}\overline{b}\overline{c}),
\end{align*}
for all $a,b, c\in H, \overline{a}, \overline{b}, \overline{c}\in
Q$. Then the coquasi-bialgebra structure is the Hopf algebra
quotient structure, and the $^Q\Mo$ is tensor equivalent to the
de-equivariantization by $Spec(K)$.
\end{proof}

The interesting point of the Proposition above is that this provides
a categorical interpretation of the tensor category $^Q\Mo$ in terms
of de-equivariantization of an affine group scheme.

\begin{exa}
Let $G$ be a connected, simply connected complex simple Lie group,
and let $\mathfrak{g}$ be its associated Lie algebra. In \cite{ArG}
the authors consider the following setting: an injective map of Hopf
algebras $\iota:\cO(G)\hookrightarrow A$, and a surjective map,
$\pi:A\to Q$, satisfying the conditions
\begin{enumerate}
\item[i)] $\pi\circ\iota(a)=\epsilon(a)1_Q$, for all $a\in
\cO(G)$;
\item[ii)] $A^{co\pi}=\cO(G)$;
\item[iii)] $\ker \pi=\cO(G)^+ A$;
\item[iv)] either $A$ is flat as $\cO(G)$-module, or the functor
$\Ind:\ ^Q\cM\to\ ^A\cM$ is exact and faithful.
\end{enumerate}
Therefore, they obtain an equivalence between the category $^Q\cM$
and the de-equivariantization of $^A\cM$ by $G$, see \cite[Thm.
2.8]{ArG}. Now, our results give an alternative proof to this
equivalence and we can state that this is a tensor equivalence.

They apply the result to the following case. Let $l \geq 3$ be an
odd integer, relative prime to 3 if $\mathfrak{g}$ contains a
$G_2$-component, and let $\zeta$ be a complex primitive $l$-th root
of $1$. By $\cO_\zeta(G)$ we denote the complex form of the
quantized coordinate algebra of $G$ at $\zeta$ and by
$u_\zeta(\mathfrak{g})$ the Frobenius-Lusztig kernel of
$\mathfrak{g}$ at $\zeta$, see \cite{DL} for definitions.

We need the following facts about $\cO_\zeta(G)$, see \cite[Prop.
6.4]{DL}: it fits into the following cocleft central exact sequence
$$ 1\to \cO(G)\to \cO_\zeta(G) \to u_\zeta(\mathfrak{g})^*\to
1.$$ Then the tensor category of modules over the Frobenius-Lusztig
kernel is a de-equivariantization of $\cO_\zeta(G)$, which is the
main result of \cite{ArG}. Moreover, the main result in \cite{AnG}
establishes that any quantum subgroup is obtained as a cocleft
central exact sequence, similar to the previous one, that is, we can
view these constructions as de-equivariantizations.

The same construction works for the restricted two parameter
(pointed) quantum group
$\widehat{u}_{\alpha,\beta}(\mathfrak{gl}_n)$ with the algebraic
group $GL_n$, where $\alpha,\beta\in \ku$ are such that
$\alpha\beta^{-1}$ is a root of unity of order $l$, and
$\alpha^l=\beta^l=1$. According to \cite[Cor. 5.3, 5.15]{Ga}, we have a
central extension of Hopf algebras
$$ 1\to \cO(GL_n)\to \cO_{\alpha,\beta}(GL_n) \to \widehat{u}_{\alpha,\beta}(\mathfrak{gl}_n)^*\to
1.$$ Therefore Proposition \ref{prop:central extensions} shows that
the category of modules over
$\widehat{u}_{\alpha,\beta}(\mathfrak{gl}_n)$ is the
de-equivariantization of the category of comodules over
$\cO_{\alpha,\beta}(GL_n)$ by $GL_n$. A similar situation holds for
any quantum subgroup of this quantum group.
\end{exa}

\subsection{A generalization of the family of algebras
$A(H,s)$}\label{subseccion ejemplo} In this Subsection we shall
assume that $\ku$ is an algebraically closed field of characteristic
zero.

Let $\Gamma$ be a finite group and $\widehat{\Gamma}=Hom(\Gamma,\ku
^*)$. We consider a finite-dimensional coradically graded pointed
Hopf algebra $H=\oplus_{n\geq 0}H_n$, with $G(H)=\Gamma$. We assume
that $H$ is generated as an algebra by $\Gamma$ and $H_1$; this is
always the case if $\Gamma$ is abelian, see \cite[Theorem 4.15]{Ang
present}. We fix a basis $x_1, \cdots, x_{\theta}$ of the space $V$
of coinvariants of $H_1$, so $H\simeq \Bc(V)\#\ku\Gamma$, where
$\Bc(V)$ is the Nichols algebra associated to $V$, and
$\Delta(x_i)=x_i\otimes g_i+1\otimes x_i$ for some $g_i\in\Gamma$.

\begin{prop}\label{prop:case pointed HA}
Let $H=\bigoplus_{n=0}^\infty H_{n}$, $\Gamma$,
$x_1,\cdots,x_\theta$ be as before. There exists a bijection between
\begin{enumerate}
  \item[(a)] central braided Hopf subalgebras $(K,r)$, and
  \item[(b)] pairs $(G, \Phi)$, where $G$ is a central subgroup of
  $\Gamma$, and $\Phi:G\to\widehat{\Gamma}$ is a morphism of group
  such that
$$\langle g', \Phi(g)\rangle =1, \qquad gx_i g^{-1}= \langle g_i,\Phi(g)\rangle x_i,$$
for all $g,g'\in G$, $1\leq i\leq \theta$.
\end{enumerate}
The correspondence is given by defining $K=\ku G$, and extending the
evaluation map $<\cdot,\Phi(\cdot)>:\Gamma\times G\to \ku$, linearly to
$H_0\otimes K$, and as zero over $H_n$, $n\geq 1$.
\end{prop}
\begin{proof}
Given a central braided Hopf subalgebra $K \subset H$, we
have that $K\subset H_0$ is commutative, so $K=\ku G$ for some subgroup
$G$ of $\Gamma$. By \eqref{casi commutativa}, $G$ is inside the
center of $\Gamma$. By \eqref{cuasi-trenza 1} and
\eqref{cuasi-trenza 2}, we have a morphism of groups $\Phi:G\to
\widehat{\Gamma}$ given by
$$ <\gamma,\Phi(g)>:= r(\gamma, g), \qquad \gamma \in \Gamma, g\in G, $$
such that $\langle g', \Phi(g)\rangle =1$ for all $g,g'\in G$. Now
by \eqref{cuasi-trenza 2} we have also that
$$ r(x_i,g)\, gg_i+ g x_i = r(g_i,g)\, x_i  g+ r(x_i,g) \, g, $$
so $r(x_i,g)=0$, and $gx_i g^{-1}= r(g_i,g)x_i$ for all $i$ and all
$g\in G$, because $H$ is graded and $g_i\neq 1$. As $H$ is generated by skew
primitive and group-like elements, we deduce that
$$ r(x,k)= 0, \qquad \mbox{for all } k\in K, x\in H_n, n\geq 1. $$

The converse is easy to prove.
\end{proof}

\begin{rmk}
Fix a set $Q\subset \Gamma$ of representatives of the right cosets
of $G$ in $\Gamma$. Note that the map $\pi:H\to K=\ku G$ given as in
Subsection \ref{subsection:baby example} over $H_0$, and extended as
$0$ over the other components, is an integral for $K$.
\end{rmk}

\begin{defi}
Let $H=\bigoplus_{n\geq0} H_{n}$ be a coradically graded
finite-dimensional Hopf algebra such that $H_0=k\Gamma$, where
$\Gamma$ is a finite group, and $H$ is generated by group-like and
skew-primitive elements. For each pair $(G,\Phi)$ as in the
Proposition \ref{prop:case pointed HA}, we shall denote by
$A(H,G,\Phi)$ the coquasi-Hopf algebra associated, constructed by
using Theorem \ref{main result}.
\end{defi}

\begin{exa}
Let $H=\oplus_{n\geq 0} H_n$ be as above, where $G(H)=\Gamma$ is an
abelian group. Therefore $H$ is generated by the group-like elements
and a finite set $x_1,\ldots , x_\theta$ of $(\gamma_i,1)$-primitive
elements, i.e.
$$\Delta(x_i)=x_i\otimes \gamma_i+1\otimes x_i, \qquad  i\in \{1,\ldots ,\theta\},$$
and also we can suppose that there are characters
$\chi_i\in\widehat{\Gamma}$ such that
$$ \gamma x_i\gamma^{-1}=\chi_i(\gamma)x_i,  \qquad  i\in \{1,\ldots ,\theta\}, \gamma\in\Gamma.$$
In this case, $\Phi$ satisfies the condition $\chi_i(\gamma)=\langle
\gamma_i, \Phi(\gamma)\rangle$ for all $i\in \{1,\ldots ,\theta\}$
and all $\gamma\in\Gamma$. Therefore $\Phi(\gamma)$ is uniquely
determined (and possibly it does not exist) when the $\gamma_i$'s
generate $\Gamma$ as a group.

The Nichols algebra $\Bc(V)$ admits a $\N^\theta$-gradation, and we
can fix a basis $B$ of $\Bc(V)$ whose elements are
$\N^\theta$-homogeneous, and such that $1\in B$. For each $\mathbf{x}\in B$
we denote $|\mathbf{x}|\in\N^\theta$ its degree, and
$$ \gamma_\mathbf{x}:=\gamma_1^{a_1}\cdots \gamma_\theta^{a_\theta} , \quad
\chi_\mathbf{x}:=\chi_1^{a_1}\cdots \chi_\theta^{a_\theta} , \qquad \mbox{if
} |\mathbf{x}|=(a_1,\ldots,a_\theta)\in\N^\theta.$$ Therefore $\gamma
\mathbf{x}\gamma^{-1}=\chi_\mathbf{x}(\gamma)\mathbf{x}$ for all $\gamma\in\Gamma$, and
$\Delta(\mathbf{x})$ is written as the sum of $\mathbf{x}\otimes 1$ plus
$\gamma_\mathbf{x}\otimes \mathbf{x}$ plus terms in intermediate degrees for the
$\N$-gradation. Fix a set of representatives elements
$q_1=e,\ldots,q_t \in \Gamma$ of $\Gamma/G$, so $Q$ has a basis
$(\overline{q_i\mathbf{x}})_{\mathbf{x}\in B, 1\leq i\leq t}$. Therefore the
multiplication and the associator of $Q$ are given by:
\begin{align*}
m(q_i\mathbf{x},q_j\mathbf{y})&=
r\big(q_i\gamma_{\mathbf{x}},\pi(q_j\gamma_{\mathbf{y}})\big)
\overline{q_i\mathbf{x}\cdot q_j\mathbf{y}} =
\Phi\big(\pi(q_j\gamma_{\mathbf{y}})\big)(q_i\gamma_{\mathbf{x}})
\overline{q_i\mathbf{x}\cdot q_j\mathbf{y}} ,
\\ \omega(q_i\mathbf{x},q_j\mathbf{y},q_k\mathbf{z})&=r\big( q_i,\pi(q_jq_k)\big)r\big( q_j,\pi(q_k)\big)\delta_{\mathbf{x},1}\delta_{\mathbf{y},1}
\delta_{\mathbf{z},1}
\end{align*}
for any $x,y,z\in B$ and $1\leq i,j,k\leq t$.

\medskip

More concretely, suppose that $\Gamma$ is a cyclic group of order $m^2$,
generated by $\gamma$, and that $x_1,\dots, x_\theta$ are
the skew-primitive elements. Thus, if $q$ is a primitive $m^2$-roof of unity,
there are unique integers $d_i, b_i$, module $m^2$, such that
$$\chi_i(\gamma)=q^{d_i}, \qquad \Delta(x_i)=x_i\otimes\gamma^{b_i}+1\otimes x_i.$$
Set $G=\langle g\rangle$, where $g=\gamma^n$, so $G\simeq\Z_n$, and $\chi\in\widehat{\Gamma}$ such that $\chi(\gamma)=q$.
A morphism $\Phi:G\to \widehat{\Gamma}$ is determined by an integer $s$ (unique modulo $n$) such that $\Phi(g)=\chi^{ns}$.
Therefore the conditions in Proposition \ref{prop:case pointed HA} are satisfied for each element in
$$ \Upsilon'(H):=\{s:\, 0\leq s\leq n-1, \quad b_is\equiv d_i (n), \forall i=1,\ldots,\theta \}. $$
A set of representatives of $\Gamma/G\simeq\Z_n$ is given by $\gamma^i$, $0\leq i\leq n-1$.
\begin{rmk}
If $H$ is a finite dimensional Hopf algebra as in this example, then
$H^*$ also is of this type and $\Upsilon'(H^*)=\Upsilon(H)$, where
$\Upsilon(H)$ was defined in \cite{A}.
\end{rmk} For each $s\in \Upsilon'(H)$ there exists a coquasi-Hopf
algebra $A'(H,s)$. We identify the group of simple (one-dimensional)
comodules with $\Z_n$, and the 3-cocycle determining the associator
is
$$ \omega(\gamma^i,\gamma^j,\gamma^k)=q^{nsi\big(j+k-(j+k)'\big)}, \qquad  0\leq i,j,k\leq n-1,$$
where $j'$ denotes the remainder of $j$ in the division by $n$. Note
that $A'(H,s)$ is dual to the quasi-Hopf algebra $A(H^*,s)$ of
\cite{A}, and these quasi-Hopf algebras include the examples in
\cite{Ge}.
\end{exa}

\begin{exa}
We consider now de-equivariantizations of some pointed Hopf algebras
related with small quantum groups by applying the previous
construction.

We fix then a finite Cartan matrix $A=(a_{ij})_{1\leq i,j \leq
\theta}$ corresponding to a semisimple Lie algebra $\mathfrak{g}$,
positive integers $d_i$, $1\leq i\leq \theta$ such that they are the
minimal ones satisfying $d_ia_{ij}=d_ja_{ji}$, and let $\Delta_+$ be
its set of positive roots and $M:= |\Delta_+|$.

Fix also a root of unity $q$ of order $N=mn$, $m,n>1$,
$q_{ij}:=q^{d_ia_{ij}}$, $\{\alpha_i\}$ the canonical basis of
$\Z^\theta$, $\chi:\Z^{\theta}\times \Z^{\theta}\to \ku^\times$ the
bicharacter determined by $\chi(\alpha_i,\alpha_j)=q_{ij}$, $1\leq
i,j\leq \theta$. Let $q_{\beta}:=\chi(\beta,\beta)$ and
$N_\beta:=\ord q_\beta$, for each $\beta\in\Delta_+$.

We will describe the corresponding Nichols algebra $\Bc(V)$ of
diagonal type attached to $(q_{ij})$ and the corresponding Hopf
algebra obtained by bosonization by a particular abelian group. We
refer to \cite[Theorems 1.25, 3.1]{Ang present} for the
corresponding statements about the Nichols algebra. Fix a basis
$x_1,\ldots,x_\theta$ of $V$, the group $\Gamma=(\Z_N)^\theta$, with
generators $\gamma_1,\ldots,\gamma_\theta$ of each cyclic group of
order $N$, and consider the realization of $V$ as a Yetter-Drinfeld
module with comodule structure determined by
$\delta(x_i)=\gamma_i\otimes x_i$. Recall that the braided adjoint
action of $x_i$ has the following property:
$$(\ad_c x_i)y:=x_iy- \chi(\alpha_i,\beta) yx_i, \qquad y\in\Bc(V) \, \Z^\theta-\mbox{homogeneous of degree }\beta.$$

The associated finite-dimensional pointed Hopf algebra
$H=\Bc(V)\#\ku\Gamma$ is described as follows. As an algebra, it is
generated by $\gamma_1,\ldots,\gamma_\theta$, $x_1,\ldots,x_\theta$,
which satisfies the following relations:
\begin{align*}
&\gamma_i^N=1, \qquad \gamma_i\gamma_j=\gamma_j\gamma_i,\qquad
&\gamma_ix_j=q_{ij}x_j\gamma_i,
\\ &(\ad_cx_i)^{1-a_{ij}}x_j=0,\quad i\neq j, \qquad &x_\beta^{N_\beta}=0,
\quad \beta\in\Delta_+,
\end{align*}
if $N\geq 8$ (otherwise we need extra relations). Each $x_\beta$ is
an homogeneous element of $\Bc(V)$ of degree $\beta$, obtained for a
fixed convex order on the roots $\beta_1<\beta_2<\cdots<\beta_M$,
and $H$ has a PBW basis $B$ as follows:
\begin{align*}
\left\{\gamma_1^{a_1}\cdots\gamma_\theta^{a_\theta}x_{\beta_M}^{b_M}
\cdots x_{\beta_1}^{b_1}: \, 0\leq a_i <N, 0\leq b_j <N_{\beta_j}
\right\}.
\end{align*}
The coproduct is determined by
$$ \Delta(\gamma_i)=\gamma_i\otimes\gamma_i, \qquad \Delta(x_i)=x_i\otimes\gamma_i+1\otimes x_i. $$

We consider $n_i, m_i\in \N$ such that $N=n_im_i$ for each $1\leq
i\leq \theta$. For each $a\in\N$, we denote by $\mu_i(a)$ the
remainder of $a$ on the division by $n_i$. Call
$g_i=\gamma_i^{n_i}$, and let $G$ be the subgroup of $\Gamma$
generated by $g_1,\ldots,g_\theta$. Therefore,
$$ G\simeq \Z_{m_1}\times\cdots\times\Z_{m_\theta},\qquad G':=\Gamma/G\simeq \Z_{n_1}\times\cdots\times\Z_{n_\theta},$$
and a set of representatives of $G'$ is given by
$\gamma_1^{a_1}\gamma_2^{a_2}$, $0\leq a_i<n_i$. With this
information we can determine $j:Q\to H$ and $\pi:H\to\ku G$ by
\begin{align*}
j(\overline{\gamma_1^{a_1}\cdots\gamma_\theta^{a_\theta}x})&
=\gamma_1^{\mu_1(a_1)}\cdots\gamma_\theta^{\mu_\theta(a_\theta)}\mathbf{x},
\\
\pi(\gamma_1^{a_1}\cdots\gamma_\theta^{a_\theta}\mathbf{x})&=
\gamma_1^{a_1-\mu_1(a_1)}\cdots\gamma_\theta^{a_\theta-\mu_\theta(a_\theta)}\varepsilon(\mathbf{x}),
\end{align*}
where $a_i\in \N$, $\mathbf{x}=x_{\beta_M}^{b_M} \cdots
x_{\beta_1}^{b_1}$. By Proposition \ref{prop:case pointed HA} we
have that
$\langle\gamma_j,\Phi(g_i)\rangle=\chi_j(g_i)=q_{ij}^{n_i}$ for
each pair $1\leq i,j\leq\theta$, so $\Phi$ is univocally determined,
and we need the extra conditions $N|n_in_j$, which is equivalent to
$m_i|n_j$, because $\langle g_j,\Phi(g_i)\rangle=1$ for all $i,j$.
To determine explicitly the coquasi-Hopf algebra structure of $Q$,
we consider the basis
$$ \overline{B}:=\left\{ \overline{\gamma_1^{a_1}\cdots\gamma_\theta^{a_\theta}\mathbf{x}}: \, 0\leq a_i<n_i, \, \mathbf{x}=x_{\beta_M}^{b_M} \cdots
x_{\beta_1}^{b_1} \right\}. $$ Given two elements
$\mathbf{x},\mathbf{y}$ of $\Bc(V)$ of degree
$(e_1,\ldots,e_\theta),(f_1,\ldots,f_\theta)\in\N_0^\theta$,
respectively, and $0\leq a_i, b_i<n_i$, we compute
\begin{align*}
m\left(\overline{\gamma_1^{a_1}\cdots\gamma_\theta^{a_\theta}\mathbf{x}},
\overline{\gamma_1^{b_1}\cdots\gamma_\theta^{b_\theta}\mathbf{y}}\right)
&= \prod_{1\leq
i,j\leq\theta}q_{ij}^{\left(b_j+f_j-\mu_j(b_j+f_j)\right)(a_i+e_i)-b_je_i}
\\ & \overline{\gamma_1^{\mu_1(a_1+b_1)}\cdots\gamma_\theta^{\mu_\theta(a_\theta+b_\theta)}\mathbf{x}\mathbf{y}},
\end{align*}
For each $\mathbf{x},\mathbf{y},\mathbf{z}\in \Bc(V)$, $0\leq a_i,
b_i,c_i<n_i$, the associator is computed as
$$ \omega\left(\overline{\gamma_1^{a_1}\cdots\gamma_\theta^{a_\theta}\mathbf{x}},
\overline{\gamma_1^{b_1}\cdots\gamma_\theta^{b_\theta}\mathbf{y}},
\overline{\gamma_1^{c_1}\cdots\gamma_\theta^{c_\theta}\mathbf{z}}\right)=
\delta_{\mathbf{x},1}\delta_{\mathbf{y},1}
\delta_{\mathbf{z},1}\prod_{1\leq
i,j\leq\theta}q_{ij}^{\left(b_j+c_j-\mu_j(b_j+c_j)\right)a_i}.$$

\medskip

Note that we can obtain the quasi-Hopf algebras appearing in
\cite{EG2} as the dual structures of the coquasi-Hopf algebras
obtained for $\Gamma=(\Z_{n^2})^\theta$ and $G\cong(\Z_n)^\theta$ as
a subgroup of $\Gamma$, i.e. $N=n^2$, $m_i=n_i=n$.
\end{exa}

\begin{exa}
Finally we consider some de-equivariantizations related with a
Nichols algebra of diagonal type but not of Cartan type. Consider a
braiding whose diagram is the last one of row 9 in \cite[Table
1]{H-classif}, and an associated $H=\Bc(V)\#\ku\Gamma$. Here we fix
$\Gamma=\Z_{9N}\times\Z_{18M}$, with generators $\gamma_1$,
$\gamma_2$ of each cyclic subgroup, respectively, and a root of
unity $q$ of order 9. Using the presentation given for the
corresponding Nichols algebra in \cite[Theorem 3.1]{Ang present}, we
can describe $H$ as follows. As an algebra, it is generated by
$\gamma_1$, $\gamma_2$, $x_1$, $x_2$, and relations
\begin{align*}
\gamma_1^{9N}&=\gamma_2^{18M}=1, \qquad
\gamma_1\gamma_2=\gamma_2\gamma_1,\qquad
\gamma_ix_j\gamma_i^{-1}=q_{ij}x_j,
\\ x_1^3&=x_2^2=x_{12}^{18}=x_{112}^{18}=x_{112}y-q yx_{112}=0,
\end{align*}
where $q_{11}=q^3$, $q_{12}=q^4=q_{21}$, $q_{11}=-1$, $(\ad_c x_i)y:=x_iy-(\gamma_i y\gamma_i^{-1})x$ for each $i=1,2$, and each $y\in\Bc(V)$, and we
consider:
\begin{align*}
& x_{12}=(\ad_cx_1)x_2, \qquad & y=x_{112}x_{12}+x_{12}x_{112},
\\ & x_{112}=(\ad_cx_1)^2x_2, \qquad & z=yx_{12}-q^2x_{12}y,
\end{align*}
so by \cite[Theorem 1.25]{Ang present} $H$ has a PBW basis $B$ as follows:
\begin{align*}
\left\{\gamma_1^{a_1}\gamma_2^{a_2}x_2^{b_6}x_{12}^{b_5}z^{b_4}y^{b_3}x_{112}^{b_2}x_1^{b_1}:
\, 0\leq a_1 <9N, 0\leq a_2 <18M, 0\leq b_2, b_5 <18,\right.
\\ \left. b_1,b_3\in\{0,1,2\}, b_4,b_6\in\{0,1\}, \right\}.
\end{align*}
The coproduct is determined by
$$ \Delta(\gamma_i)=\gamma_i\otimes\gamma_i, \qquad \Delta(x_i)=x_i\otimes\gamma_i+1\otimes x_i. $$

Fix $n, n_1, m, m_1\in \N$ such that $9N=nn_1$ and $18M=mm_1$. For
each $a\in\N$, we denote by $a'$ (respectively, $a''$) the remainder
on the division by $n$ (respectively, $m$). Call $g_1=\gamma_1^n$,
$g_2=\gamma_2^m$, and $G$ the subgroup of $\Gamma$ generated by
$g_1$ and $g_2$. Therefore,
$$ G\simeq \Z_{n_1}\times\Z_{m_1},\qquad G':=\Gamma/G\simeq \Z_{n}\times\Z_{m},$$
and a set of representatives of $G'$ are
$\gamma_1^{a_1}\gamma_2^{a_2}$, $0\leq a_1<n$, $0\leq a_2<m$. Also,
we can write explicitly:
$$ j(\overline{\gamma_1^{a_1}\gamma_2^{a_2}\mathbf{x}})=\gamma_1^{a_1'}\gamma_2^{a_2''}\mathbf{x},
\quad \pi(\gamma_1^{a_1}\gamma_2^{a_2}\mathbf{x})=
\gamma_1^{a_1-a_1'}\gamma_2^{a_2-a_2''}\varepsilon(\mathbf{x})\in G,
\qquad a_i\in \N, \mathbf{x}\in B. $$ By Proposition \ref{prop:case
pointed HA} we have that
\begin{align*}
&\langle \gamma_1,\Phi(g_1)\rangle=\chi_1(g_1)=q_{11}^n,
&\langle \gamma_2,\Phi(g_1)\rangle=\chi_2(g_1)=q_{12}^n,\\
&\langle \gamma_1,\Phi(g_2)\rangle=\chi_1(g_2)=q_{21}^m,
&\langle \gamma_2,\Phi(g_2)\rangle=\chi_2(g_2)=q_{22}^m,
\end{align*} so $\Phi$ is univocally determined, and moreover it tells us that
$m,n$ should satisfy $3|n$, $2|m$, $9|mn$, because $\langle
g_j,\Phi(g_i)\rangle=1$ for all $i,j\in\{1,2\}$.

We compute the structure of the coquasi-Hopf algebra associated to
this datum. Note that the following set is a basis of $Q$:
$$ \overline{B}:=\left\{ \overline{\gamma_1^{a_1}\gamma_2^{a_2}\mathbf{x}}: \, 0\leq a_1<n,\, 0\leq a_2<m, \, \mathbf{x}\in B\, \right\}. $$
Given $\mathbf{x},\mathbf{y}\in B$ of degree
$(e_1,e_2),(f_1,f_2)\in\N_0^2$, respectively, and $0\leq a_1,
b_1<n$, $0\leq a_2,b_2<m$, we have that
\begin{align*}
m(\overline{\gamma_1^{a_1}\gamma_2^{a_2}\mathbf{x}},
\overline{\gamma_1^{b_1}\gamma_2^{b_2}\mathbf{y}}) &=
q_{11}^{-b_1e_1}q_{12}^{(b_1+f_1-(b_1+f_1)')(a_2+e_2)-b_1e_2}
\\ & q_{21}^{(b_2+f_2-(b_2+f_2)'')(a_1+e_1)-b_2e_1}q_{22}^{-b_2e_2}
\overline{\gamma_1^{(a_1+b_1)'}\gamma_2^{(a_2+b_2)''}\mathbf{x}\mathbf{y}},
\end{align*}
where we use that $3|n$, $2|m$. And the associator is given by
$$ \omega(\overline{\gamma_1^{a_1}\gamma_2^{a_2}\mathbf{x}}, \overline{\gamma_1^{b_1}\gamma_2^{b_2}\mathbf{y}},
\overline{\gamma_1^{c_1}\gamma_2^{c_2}\mathbf{z}})=
\delta_{\mathbf{x},1}\delta_{\mathbf{y},1}
\delta_{\mathbf{z},1}q_{12}^{a_2(b_1+c_1-(b_1+c_1)')}q_{21}^{a_1(b_2+c_2-(b_2+c_2)'')},
$$ where $\mathbf{x},\mathbf{y},\mathbf{z}\in B$, $0\leq a_1,
b_1,c_1<n$, $0\leq a_2,b_2,c_2<m$.
\end{exa}

\end{document}